\documentclass[10pt, reqno]{amsart}
\usepackage{eurosym}
\usepackage{amsmath, amsthm, amscd, amsfonts, amssymb, graphicx, color}
\usepackage{xcolor}
\usepackage[colorlinks=true, linkcolor=red, citecolor=red, urlcolor=red]{hyperref}

\newtheorem{theorem}{Theorem}[section]
\newtheorem{lemma}[theorem]{Lemma}

\theoremstyle{definition}

\newtheorem{example}[theorem]{Example}

\theoremstyle{remark}
\newtheorem{remark}[theorem]{Remark}
\numberwithin{equation}{section}

\begin{document}
\title[Coefficient Difference]{Invariance of the Initial Coefficient
Differences of Ma-Minda Convex Functions}
\author[ U.Raza, M. Raza, R. Ali]{ Umar Raza$^{1}$, Mohsan Raza$^{2,\ast }$, Rashid
Ali$^{2}$}
\address{$^{1}$Department of Mathematics, University of Jhang, Pakistan.}
\email{umarraza55@gmail.com }
\address{$^{2}$Department of Mathematics, Government College University
Faisalabad, Pakistan.}
\email{mohsan976@yahoo.com}
\email{rashidali5277@gmail.com}
\date{Received June 11, 2026.\\
\indent$^{\ast }$ Corresponding author\\
2020\textit{\ Mathematics Subject Classification. }30C45, 30C50.}

\begin{abstract}
Let $\Phi $ be a univalent function in $\mathbb{D}=\{z\in \mathbb{C}:|z|<1\}$%
, $\Phi (\mathbb{D})$ is symmetric with respect to the real axis, starlike
with respect to $\Phi (0)=1$, and $\Phi ^{\prime }(0)>0$. Let $\mathcal{C}%
(\Phi )$ denote the class of Ma-Minda convex functions. In this article, we
present the bounds on $||a_{3}|-|a_{2}||$ for Taylor's coefficients of the
function $f$ in the class $\mathcal{C}(\Phi )$. We also establish the same
bounds for the inverse coefficients. All the bounds we study here are sharp.
We also present the conditions such that the bounds on $|a_{3}|-|a_{2}||$ and 
$|A_{3}|-|A_{2}||$ are invariant, where $A_{2}$ and $A_{3}$ are the first
two coefficients of the Taylor series of the inverse functions of $f\in 
\mathcal{C}(\Phi ).$ Thus provides examples of invariance and nonvariance
among the subclasses of convex functions.
\end{abstract}

\maketitle


\setcounter{page}{1} 

\section{Introduction}

Let $\mathcal{A}$ denote the class of analytic functions $f$ defined in the
open unit disk $\mathbb{D}=\{z\in \mathbb{C}:|z|<1\}$, normalized by the
conditions $f(0)=0$ and $f^{\prime }(0)=1$, and having the Taylor expansion
of the form 
\begin{equation}
f(z)=z+\sum_{n=2}^{\infty }a_{n}z^{n}.  \label{1}
\end{equation}%
Let $\mathcal{S}$ represent a subclass of functions in $\mathcal{A}$ that
are univalent(one-to-one) in $\mathbb{D}$. In 1985, de Branges \cite{c2}
settled the celebrated Bieberbach conjecture by proving that for every
function $f\in \mathcal{S}$ given in the form $\left( \ref{1}\right) $, the
sharp estimate $|a_{n}|\leq n,n\geq 2,$ holds, with equality attained only
by the Koebe function $k(z)=z/(1-z)^{2}$ or its rotations. This remarkable
breakthrough naturally led to further investigations concerning the behavior
of successive coefficients of univalent functions. In particular, it became
interesting to ask whether the inequality 
\begin{equation*}
\big||a_{n+1}|-|a_{n}|\big|\leq 1,\qquad n\geq 2,
\end{equation*}%
is valid for all $f\in \mathcal{S}$. However, it was soon observed that this
is not true even for $n=2$. In fact, it was shown in \cite{c1} that the
following sharp bounds hold: 
\begin{equation*}
-1\leq |a_{3}|-|a_{2}|\leq \frac{3}{4}+e^{-\lambda _{0}}\left( 2e^{-\lambda
_{0}}-1\right) =1.029\cdots ,
\end{equation*}%
where $\lambda _{0}$ is the unique root in $0<\lambda <1$ of the equation $%
4\lambda =e^{\lambda }.$ Later, Hayman \cite{c3} established an important
general result by proving that there exists an absolute constant $C>0$ such
that 
\begin{equation*}
\big||a_{n+1}|-|a_{n}|\big|\leq C,\qquad n\geq 2,
\end{equation*}%
for all functions $f\in \mathcal{S}$. Hayman's original proof relied on his
powerful method developed for the study of areally mean $p$-valent functions 
\cite{c5}. An alternative approach was later provided by Milin through the
celebrated Lebedev--Milin inequalities, and a detailed exposition of these
developments can be found in Duren's classical monograph \cite{c1}. Despite
the significance of Hayman's estimate, little progress has been achieved in
determining sharper bounds for $C$. Ilina \cite{lina} proved in 1968 that $%
C<4.26\cdots $, and subsequently Grispan \cite{c4}, refining Milin's method,
improved this in 1976 by establishing that for $n\geq 2$, 
\begin{equation*}
-2.97\cdots <|a_{n+1}|-|a_{n}|<3.61\cdots .
\end{equation*}%
For many years, no further improvements were reported until the recent
contribution of Obradovi\'{c} et al \cite{ob}, who employed Grunsky
inequalities to obtain a sharper estimate for the case $n=3$, namely, 
\begin{equation*}
|a_{4}|-|a_{3}|\leq 2.1033\cdots .
\end{equation*}

Thus, except for the sharp bounds known in the case $n=2$, no sharp upper or
lower estimates are currently available for the difference $%
|a_{n+1}|-|a_{n}| $ when $n\geq 3$ for functions belonging to the class $%
\mathcal{S}$.

The coefficient problem for some subclasses of the class $\mathcal{S}$ is
settled by several authors. We first define some well-known subclasses of
the class $\mathcal{S}$.

An analytic function $f$ is called \emph{starlike} if the image domain $f(%
\mathbb{D})$ is starlike with respect to the origin. The class of all
univalent starlike functions is denoted by $\mathcal{S}^{\ast }$.
Analytically, a function $f$ belongs to $\mathcal{S}^{\ast }$ if and only if 
\begin{equation*}
\Re \left( \frac{zf^{\prime }(z)}{f(z)}\right) >0,\qquad z\in \mathbb{D}.
\end{equation*}%
Similarly, an analytic function $f$ is called \emph{convex} if the image
domain $f(\mathbb{D})$ is convex. The class of all univalent convex
functions is denoted by $\mathcal{C}$. Analytically, a function $f$ belongs
to $\mathcal{C}$ if and only if 
\begin{equation*}
\Re \left( 1+\frac{zf^{\prime \prime }(z)}{f^{\prime }(z)}\right) >0,\qquad
z\in \mathbb{D}.
\end{equation*}%
A function $f$ is said to be \emph{close-to-convex }if and only if there
exists $g\in \mathcal{S}^{\ast }$ such that 
\begin{equation*}
\Re \left( \frac{zf^{\prime }(z)}{g(z)}\right) >0,\qquad z\in \mathbb{D}.
\end{equation*}%
The class of all close-to-convex functions in $\mathbb{D}$ is denoted by $%
\mathcal{K}$. In 1973, Pommerenke \cite{c5} conjectured that for $f\in 
\mathcal{S}^{\ast }$, $\bigl||a_{n+1}|-|a_{n}|\bigr|\leq 1,\ n\geq 2,$ and
in 1978, Leung \cite{leu} proved this conjecture by showing that equality is
attained for the function 
\begin{equation*}
f(z)=\frac{z}{(1-\rho z)(1-\sigma z)},\qquad |\rho |=|\sigma |=1.
\end{equation*}
In 1985, Koepf \cite{koepf} showed that for $f\in \mathcal{K}$, $\bigl|%
|a_{n+1}|-|a_{n}|\bigr|\leq 1$ for $n=2$ and it is an open problem for $%
n\geq 3.$ It seems to be a challenging problem to find the sharp upper and
lower bounds for $\bigl||a_{n+1}|-|a_{n}|\bigr|$, when $f\in \mathcal{C}$,
that is, for the convex functions. The only noteworthy results to date are
attributed to Ming and Sugawa \cite{Ming}, where sharp upper bounds have
been discovered when $n\geq 2$ and has sharp lower bounds when $n=2,3$.
Determining the sharp lower bounds for $n\geq 4$ is an open problem. By
following this trend many other authors established the bounds on the
difference of initial coefficients of different subclasses of univalent
functions. Cho et al. \cite{cho} presented the bounds on $||a_{3}|-|a_{2}||$
for the class of Bazilevic functions and for the class of non-Bazilevi\v{c}
functions the bounds on $||a_{3}|-|a_{2}||$ were recently calculated in \cite%
{ali}. Moreover, for the class $\mathcal{U}(\alpha ,\lambda )$ the bounds on
the difference of initial coefficients were established in \cite{Raza}.

Now, let $\mathcal{B}$ represent the class of all analytic (holomorphic)
functions $\omega $ in $\mathbb{D}$ with the property that $\omega (0)=0$
and $|\omega (z)|<1$ for $z\in \mathbb{D}$. These functions are called
Schwarz functions. A number of problems in geometric function theory can be
answered in an easy and precised way by using the concept of subordination.
An analytic function $f$ is said to be subordinate to some other analytic
function $g$ if there exists $\omega \in \mathcal{B}$ such that $%
f(z)=g\left( \omega (z)\right) $ for $z\in \mathbb{D}$. In the case, if $g$
is univalent and $f(0)=g(0)$, then $f(\mathbb{D})\subset g(\mathbb{D})$.

In order to unify and generalize many well-known subclasses of convex and
starlike functions, Ma and Minda \cite{11} introduced a broad family of
analytic function classes associated with a suitable analytic function $\Phi 
$. Let $\Phi $ be an analytic and univalent function in $\mathbb{D}$ such
that $\Phi (\mathbb{D})$ is symmetric with respect to the real axis,
starlike with respect to $\Phi (0)=1$, and satisfies $\Phi ^{\prime }(0)>0$.
The class of Ma-Minda convex functions, denoted by $\mathcal{C}(\Phi )$,
consists of those functions $f\in \mathcal{A}$ for which 
\begin{equation}
1+\frac{zf^{\prime \prime }(z)}{f^{\prime }(z)}\prec \Phi (z),\qquad z\in 
\mathbb{D},
\end{equation}%
where the function $\Phi (z)$ has series expansion of the
form 
\begin{equation*}
\Phi (z)=1+B_{1}z+B_{2}z^{2}+B_{3}z^{3}+\ldots .
\end{equation*}%
%
%
%
%
%
%
%
%
%
%
The Koebe's $1/4$ theorem states that there exists an inverse function $%
f^{-1}$ for every univalent function $f$ defined in $\mathbb{D}$, at least on
the disk with a radius 1/4, having the series expansion as follows%
\begin{equation}
f^{-1}(z)=z+\sum_{n=2}^{\infty }A_{n}z^{n}.  \label{equ34}
\end{equation}%
Since $f(f^{-1}(z))=z$, so from (\ref{1}) and (\ref{equ34}), we have 
\begin{equation}
A_{2}=-a_{2},\ \ \ \text{and}\ \ \ A_{3}=2a_{2}^{2}-a_{3}.  \label{A2}
\end{equation}%
Libera and Zlotkiewicz \cite{lib1} were the first to demonstrate that, for
the class $\mathcal{C}$ of convex functions, the inverse coefficients $A_{n}$
retain the traditional inequality $\left\vert a_{n}\right\vert \leq 1$ when $%
2\leq n\leq 7$. Some invariance features among the family of strongly convex
functions were demonstrated by Thomas and Verma \cite{DK} in 2016. In \cite%
{raza4}, the invariance property between the coefficient functionals of the
subclass of convex functions associated with sigmoid functions is examined.
Thomas recently provided a thorough explanation of this characteristic for a
few coefficient functions for the class of convex functions and its
subclasses. This characteristic is questioned for the subclass of convex
functions associated with the cardioid domain, see \cite{raza3}.

Many authors have recently explored coefficient bounds for inverse functions
(see \cite{m30,m37}). Specifically, Sim and Thomas showed that $-1\leq
|A_{3}|-|A_{2}|\leq 3$ for $f\in \mathcal{S}$, \cite{m30}. They also studied
the sharp bounds on coefficient differences for some other subclasses of
univalent functions.\newline
In this article, we present the bounds on $||a_{3}|-|a_{2}||$ for the class $%
\mathcal{C}(\Phi )$. We also establish the bounds on $||A_{3}|-|A_{2}||$.
The bounds being presented here are sharp.

\section{Main Results}

The Caratheodory class, denoted by $\mathcal{P}$, is the collection of
holomorphic functions $p$ in the unit disk $\mathbb{D}=\left\{ z\in \mathbb{C%
}:|z|<1\right\} $ satisfying the condition $Re\left( p(z)\right) >0$ for $%
z\in \mathbb{D}$ and having series expansion of the form 
\begin{equation*}
p(z)=1+\sum_{n=1}^{\infty }p_{n}z^{n}.
\end{equation*}%
We use the following lemma to prove our results.

\begin{lemma}
\label{lemma1}\cite{19} Let $p\in \mathcal{P}$. Then 
\begin{equation*}
p_{1}=2t_{1},
\end{equation*}%
and 
\begin{equation*}
p_{2}=2t_{1}^{2}+2(1-|t_{1}|^{2})t_{2},
\end{equation*}%
where $t_{i}\in \overline{\mathbb{D}}$ for $i\in \left\{ 1,2\right\} $. If $%
|t_{1}|=1$, then there exists a unique function $p\in \mathcal{P}$ given as 
\begin{equation*}
p(z)=\frac{1+t_{1}z}{1-t_{1}z}.
\end{equation*}%
If $t_{1}\in \mathbb{D}$ and $|t_{2}|=1$, then there exists a unique
function $p\in \mathcal{P}$ defined as 
\begin{equation}
p(z)=\frac{1+(\overline{t_{1}}t_{2}+t_{1})z+t_{2}z^{2}}{1+(\overline{t_{1}}%
t_{2}-t_{1})z-t_{2}z^{2}}.  \label{eq001}
\end{equation}
\end{lemma}

In our first result, we establish the sharp bound on the difference of
initial coefficients $||a_{3}|-|a_{2}||$.

\begin{theorem}
\label{Th1}Let $f\in \mathcal{C}(\Phi )$ be of the form \eqref{1} with $%
B_{1}>0$. Then 
\begin{equation*}
|a_{3}|-|a_{2}|\leq 
\begin{cases}
\frac{\left\vert B_{1}^{2}+B_{2}\right\vert -3B_{1}}{6}, & \text{if }\
|B_{1}^{2}+B_{2}|\geq 4B_{1}, \\ 
\frac{B_{1}}{6}, & \text{if }\ |B_{1}^{2}+B_{2}|<4B_{1},%
\end{cases}%
\end{equation*}%
and 
\begin{equation*}
|a_{3}|-|a_{2}|\geq 
\begin{cases}
\frac{\left\vert B_{1}^{2}+B_{2}\right\vert -3B_{1}}{6}, & \text{if }\
B_{1}\geq 2|B_{1}^{2}+B_{2}|, \\[8pt] 
-\frac{B_{1}}{2}\sqrt{\frac{B_{1}}{B_{1}+|B_{1}^{2}+B_{2}|}}, & \text{if }\
5B_{1}\leq 4|B_{1}^{2}+B_{2}|, \\[8pt] 
-\frac{B_{1}}{6}-\frac{3B_{1}^{2}}{8{|B_{1}^{2}+B_{2}|}+8{B_{1}}}, & \text{%
otherwise}.%
\end{cases}%
\end{equation*}%
These inequalities are sharp.
\end{theorem}

\begin{proof}
Since $f\in \mathcal{C}(\Phi )$, then there exists a
function $w\in \mathcal{B}$ such that 
\begin{equation}
1+\frac{zf^{\prime \prime }(z)}{f^{\prime }(z)}=\Phi \left( w(z)\right) .
\label{ma}
\end{equation}%
Let $p\in \mathcal{P}$. Then 
\begin{equation}
p(z)=\frac{1+w(z)}{1-w(z)}=1+p_{1}z+p_{2}z^{2}+p_{3}z^{3}+\cdots .  \label{p}
\end{equation}%
Now, by comparing the coefficients of \eqref{ma} and \eqref{p}, we have 
\begin{equation}
a_{2}=\frac{B_{1}}{4}p_{1},\quad \quad a_{3}:=\left( \frac{%
-B_{1}+B_{1}^{2}+B_{2}}{24}\right) p_{1}^{2}+\frac{B_{1}}{12}p_{2}.
\label{a2}
\end{equation}%
By using \eqref{a2}, we get 
\begin{equation*}
|a_{3}|-|a_{2}|=|bp_{1}^{2}+cp_{2}|-|ap_{1}|,
\end{equation*}%
with 
\begin{equation*}
a:=\frac{B_{1}}{4},\quad \quad b:=\left( \frac{-B_{1}+B_{1}^{2}+B_{2}}{24}%
\right) ,\quad \text{and}\quad c:=\frac{B_{1}}{12}.
\end{equation*}%
By using Lemma \ref{lemma1}, we get 
\begin{align*}
|a_{3}|-|a_{2}|& ={|4bt_{1}^{2}+c(2t_{1}^{2}+2(1-{|t_{1}|}^{2})t_{2})|}-2{%
|at_{1}|} \\
& \leq {m|t_{1}|}^{2}+2{c}\left( 1-{|t_{1}|}^{2}\right) |t_{2}|-2{|at_{1}|},
\end{align*}%
where $m:=|4b+2c|$. Since $\mathcal{P}$ is rotationally invariant, so we can
assume, $t_{1}\in \lbrack 0,1]$ and by using the fact that $|t_{2}|\leq 1$,
we have 
\begin{equation*}
|a_{3}|-|a_{2}|\leq \left( {m}-2{c}\right) t_{1}^{2}-2{a}t_{1}+2{c}=\Upsilon
(x),
\end{equation*}%
where 
\begin{equation}
\Upsilon (x)=k_{2}x^{2}+k_{1}x+k_{0},  \label{equ2}
\end{equation}%
with 
\begin{equation*}
k_{2}={m}-2{c},\quad \quad k_{1}=-2{a},\quad \quad k_{0}=2{c}.
\end{equation*}%
Here we have two cases:\newline
$\mathbf{A}_{1}:$ If $k_{2}\leq 0$. Given that $k_{1}<0$, thus by (\ref{equ2}
) we get, $\Upsilon ^{\prime }(x)=2k_{2}x+k_{1}<0$ and it yields that $%
\Upsilon $ is a decreasing function. Therefore 
\begin{equation*}
\Upsilon (x)\leq \Upsilon (0)=k_{0}=2c=\frac{B_{1}}{6}.
\end{equation*}%
$\mathbf{A}_{2}:$ Now if $k_{2}>0$, then $\Upsilon $ is a quadratic function
and it has positive leading coefficient so, 
\begin{equation*}
\Upsilon (x)\leq \max \{\Upsilon (0),\Upsilon (1)\}.
\end{equation*}%
Let $k_{2}+k_{1}<0$. Then $\Upsilon (1)=k_{2}+k_{1}+k_{0}<k_{0}=\Upsilon (0)$
and thus $\Upsilon (x)\leq \Upsilon (0)=2c=\frac{B_{1}}{6}$. Further, let $%
k_{2}+k_{1}\geq 0$. Then $\Upsilon (1)=k_{2}+k_{1}+k_{0}\geq k_{0}=\Upsilon
(0)$. Thus 
\begin{equation*}
\Upsilon (x)\leq \Upsilon (1)=k_{2}+k_{1}+k_{0}=|4b+2c|-2a=\frac{\left\vert
B_{1}^{2}+B_{2}\right\vert -3B_{1}}{6}.
\end{equation*}%
From above discussion, we have 
\begin{equation*}
|a_{3}|-|a_{2}|\leq 
\begin{cases}
\frac{\left\vert B_{1}^{2}+B_{2}\right\vert -3B_{1}}{6}, & \text{if }\
|B_{1}^{2}+B_{2}|\geq 4B_{1}, \\ 
\frac{B_{1}}{6}, & \text{if }\ |B_{1}^{2}+B_{2}|<4B_{1}.%
\end{cases}%
\end{equation*}%
The bound is sharp for the function $f$ defined by \eqref{ma} with $%
p(z)=(1+z)/(1-z)$, when $|B_{1}^{2}+B_{2}|\geq 4B_{1}$ and for $%
|B_{1}^{2}+B_{2}|<4B_{1}$ the equality exists with $p(z)=(1+z^{2})/(1-z^{2})$%
.\newline
Now, we will find the lower bound on $|a_{3}|-|a_{2}|$.\newline
By using Lemma \ref{lemma1}, we get 
\begin{equation*}
|a_{3}|-|a_{2}|=\left\vert me^{i\theta }t_{1}^{2}+2ce^{i\phi
_{1}}(1-|t_{1}|^{2})t_{2}\right\vert -2|at_{1}|,
\end{equation*}%
where 
\begin{equation*}
m=\left\vert 4b+2c\right\vert ,\quad \quad \theta =arg(4b+2c),\quad \quad
\phi _{1}=arg(2c).
\end{equation*}%
Since $\mathcal{P}$ is rotationally invariant so we can assume, $t_{1}\in
\lbrack 0,1]$ and $t_{2}=re^{i\phi _{2}}$ with $r\in \lbrack 0,1]$. Thus 
\begin{equation*}
|a_{3}|-|a_{2}|=\left\vert me^{i\theta }t_{1}^{2}+2rce^{i\phi
}(1-t_{1}^{2})\right\vert -2at_{1},
\end{equation*}%
with $\phi =\phi _{1}+\phi _{2}$. As $\left\vert e^{i\phi }\right\vert =1$,
so 
\begin{equation*}
|a_{3}|-|a_{2}|=\left\vert me^{i(\theta -\phi
)}t_{1}^{2}+2rc(1-t_{1}^{2})\right\vert -2at_{1}.
\end{equation*}%
Now 
\begin{align}
|a_{3}|-|a_{2}|=& \sqrt{m^{2}t_{1}^{4}+4mcrt_{1}^{2}(1-t_{1}^{2})\cos
(\theta -\phi )+4c^{2}r^{2}(1-t_{1}^{2})^{2}}-2at_{1},  \notag \\
\geq & \left\vert mt_{1}^{2}-2rc(1-t_{1}^{2})\right\vert -2at_{1},
\label{equ3}
\end{align}%
as $\cos (\theta -\phi )\geq -1$. If $b=c=0$, then 
\begin{equation*}
|a_{3}|-|a_{2}|\geq -2at_{1}\geq -2a=-\frac{B_{1}}{2}.
\end{equation*}%
Now, let $b\neq 0,\ c\neq 0$. Then from \eqref{equ3}, we have two cases: 
\newline
\textbf{B}$_{1}$: If $mt_{1}^{2}-2rc(1-t_{1}^{2})\leq 0$, then 
\begin{equation*}
t_{1}\leq \sqrt{\frac{2rc}{m+2rc}}:=\zeta _{1}.
\end{equation*}%
So, from (\ref{equ3}) and by using the fact, $r\in \lbrack 0,1]$, we have 
\begin{align*}
|a_{3}|-|a_{2}|& \geq -(m+2rc)t_{1}^{2}-2at_{1}+2rc \\
& \geq -(m+2rc)\zeta _{1}^{2}-2a\zeta _{1}+2rc \\
& =-2a\sqrt{\frac{2rc}{m+2rc}}=-2a\sqrt{\frac{r}{\frac{m}{2c}+r}} \\
& \geq -2a\sqrt{\frac{1}{\frac{m}{2c}+1}} \\
& ={-2a}\sqrt{\frac{2c}{m+2c}} \\
& =-\frac{B_{1}}{2}\sqrt{\frac{B_{1}}{B_{1}+|B_{1}^{2}+B_{2}|}}.
\end{align*}%
\textbf{B}$_{2}$: If $mt_{1}^{2}-2rc(1-t_{1}^{2})>0$, then $t_{1}>\zeta _{1}$%
, and we define 
\begin{equation*}
\varphi (t_{1},r)=(m+2rc)t_{1}^{2}-2at_{1}-2rc.
\end{equation*}%
Since, for $a^{2}\leq 2c(m+2c)$ and by using the fact that $r\in \lbrack
0,1] $, we have 
\begin{equation*}
\frac{\partial \varphi }{\partial t_{1}}=2(m+2r{c})t_{1}-2{a}\geq 2\sqrt{%
2rc(m+2rc)}-2a\geq 0.
\end{equation*}%
It implies, $\varphi $ is an increasing function and from (\ref{equ3}), we
have 
\begin{equation*}
|a_{3}|-|a_{2}|\geq -2{a}\sqrt{\frac{2r{c}}{m+2r{c}}}\geq -2{a}\sqrt{\frac{2{%
c}}{m+2{c}}}=-\frac{B_{1}}{2}\sqrt{\frac{B_{1}}{B_{1}+|B_{1}^{2}+B_{2}|}},
\end{equation*}%
as $r\in \lbrack 0,1]$. Now, for $a\geq m+2c$, we deduce 
\begin{equation*}
\frac{\partial \varphi }{\partial t_{1}}=2(m+2r{c})t_{1}-2{a}\leq 2m+4r{c}-2{%
a}\leq 2m+4{c}-2{a}\leq 0.
\end{equation*}%
It implies, $\varphi $ is a decreasing function and from (\ref{equ3}), we
get 
\begin{equation*}
|a_{3}|-|a_{2}|\geq m-2{a}=\frac{\left\vert B_{1}^{2}+B_{2}\right\vert
-3B_{1}}{6}.
\end{equation*}%
At the end, for $a^{2}>2c(m+2c)$ and $a<m+2c$, we have 
\begin{equation*}
\zeta _{1}<\frac{{a}}{m+2{c}}:=\zeta _{2}<1,
\end{equation*}%
and $\frac{\partial \varphi }{\partial r}=-2c(1-t_{1}^{2})\leq 0$ for $r\in
\lbrack 0,1]$ and $t_{1}\in \lbrack \zeta _{1},1]$. Thus 
\begin{equation*}
\varphi (t_{1},r)\geq \varphi (t_{1},1)=(m+2{c})t_{1}^{2}-2{a}t_{1}-2{c}%
=h(t_{1}).
\end{equation*}%
Since $h^{\prime }(t_{1})=0$ yields $t_{1}=\zeta _{2}$. Thus from (\ref{equ3}
) 
\begin{equation*}
|a_{3}|-|a_{2}|\geq -2{c-\frac{{a}^{2}}{m+2{c}}}=-\frac{B_{1}}{6}-\frac{%
3B_{1}^{2}}{8{|B_{1}^{2}+B_{2}|}+8{B_{1}}}.
\end{equation*}%
From above discussion, we have 
\begin{equation*}
|a_{3}|-|a_{2}|\geq 
\begin{cases}
\frac{\left\vert B_{1}^{2}+B_{2}\right\vert -3B_{1}}{6}, & \text{if }\
B_{1}\geq 2|B_{1}^{2}+B_{2}|, \\[8pt] 
-\frac{B_{1}}{2}\sqrt{\frac{B_{1}}{B_{1}+|B_{1}^{2}+B_{2}|}}, & \text{if }\
5B_{1}\leq 4|B_{1}^{2}+B_{2}|, \\[8pt] 
-\frac{B_{1}}{6}-\frac{3B_{1}^{2}}{8{|B_{1}^{2}+B_{2}|}+8{B_{1}}}, & \text{%
otherwise}.%
\end{cases}%
\end{equation*}%
When $B_{1}\geq 2|B_{1}^{2}+B_{2}|$, the inequality is sharp for the
function $f$ defined by \eqref{ma} with $p(z)=(1+z)/(1-z)$.\newline
Let $5B_{1}\leq 4|B_{1}^{2}+B_{2}|$. Then by Lemma \ref{lemma1}, the
equality exists for the function $f$ defined by \eqref{ma}, where $p(z)$ is
given by \eqref{eq001} with 
\begin{equation*}
t_{1}=\sqrt{\frac{{B_{1}}}{{|B_{1}^{2}+B_{2}|}+{B_{1}}}},\quad \text{and}%
\quad t_{2}=%
\begin{cases}
-\frac{B_{1}^{2}+B_{2}}{{|B_{1}^{2}+B_{2}|}}, & \text{if }\
B_{1}^{2}+B_{2}\neq 0, \\[1.2em] 
1, & \text{if }\ B_{1}^{2}+B_{2}=0.%
\end{cases}%
\end{equation*}%
For $B_{1}^{2}+B_{2}=0$, the calculations are obvious and $|a_{3}|-|a_{2}|=-%
\frac{B_{1}}{2}$.\newline
Now for $B_{1}^{2}+B_{2}\neq 0$, 
\begin{align*}
|a_{3}|-|a_{2}|& =\left\vert bp_{1}^{2}+cp_{{2}}|-|ap_{1}\right\vert \\
& =\frac{1}{6}\left\vert \left( B_{1}^{2}+B_{2}\right)
t_{1}^{2}+B_{1}(1-t_{1}^{2})t_{2}\right\vert -\frac{1}{2}\left\vert
B_{1}t_{1}\right\vert ,
\end{align*}%
and after some simplifications, we get 
\begin{equation*}
|a_{3}|-|a_{2}|=-\frac{B_{1}}{2}\sqrt{\frac{B_{1}}{B_{1}+|B_{1}^{2}+B_{2}|}}.
\end{equation*}%
Lastly, for $5B_{1}>4|B_{1}^{2}+B_{2}|$ and $B_{1}<2|B_{1}^{2}+B_{2}|$, the
inequality is sharp for the function $f$ defined by \eqref{ma}, where $p(z)$
is given by \eqref{eq001} with 
\begin{equation*}
t_{1}=\frac{3{B_{1}}}{2{|B_{1}^{2}+B_{2}|}+2{B_{1}}},\quad \text{and}\quad
t_{2}=-\frac{B_{1}^{2}+B_{2}}{{|B_{1}^{2}+B_{2}|}}.
\end{equation*}%
Now 
\begin{align*}
36|bp_{1}^{2}+cp_{2}|^{2}& =\left\vert \left( B_{1}^{2}+B_{2}\right)
t_{1}^{2}+B_{1}(1-t_{1}^{2})t_{2}\right\vert ^{2} \\
& =\left\vert B_{1}^{2}+B_{2}\right\vert ^{2}t_{1}^{4}+2Re\left( B_{1}(%
\overline{B_{1}^{2}+B_{2}})t_{1}^{2}(1-t_{1}^{2})t_{2}\right)
+B_{1}^{2}(1-t_{1}^{2})^{2} \\
& =\left\vert B_{1}^{2}+B_{2}\right\vert ^{2}t_{1}^{4}-2\left\vert
B_{1}^{3}+B_{1}B_{2}\right\vert
t_{1}^{2}(1-t_{1}^{2})+B_{1}^{2}(1-t_{1}^{2})^{2} \\
& =\left( \left\vert B_{1}^{2}+B_{2}\right\vert
t_{1}^{2}-B_{1}(1-t_{1}^{2})\right) ^{2} \\
& =\left( \frac{9B_{1}^{2}}{4{|B_{1}^{2}+B_{2}|}+4{B_{1}}}-B_{1}\right) ^{2}.
\end{align*}%
Since $5B_{1}>4|B_{1}^{2}+B_{2}|$, it implies $\frac{9B_{1}^{2}}{4{%
|B_{1}^{2}+B_{2}|}+4{B_{1}}}-B_{1}>0$ and thus 
\begin{equation*}
|bp_{1}^{2}+cp_{2}|=\frac{3B_{1}^{2}}{8{|B_{1}^{2}+B_{2}|}+8{B_{1}}}-\frac{%
B_{1}}{6}.
\end{equation*}%
Therefore 
\begin{equation*}
|a_{3}|-|a_{2}|=|bp_{1}^{2}+cp_{2}|-|ap_{1}|=-\frac{B_{1}}{6}-\frac{%
3B_{1}^{2}}{8{|B_{1}^{2}+B_{2}|}+8{B_{1}}}.
\end{equation*}%
It completes the proof.
\end{proof}

In this theorem, we present the sharp bounds on $||A_{3}|-|A_{2}||$.

\begin{theorem}
\label{Th2}Let $f\in \mathcal{C}(\Phi )$ be of the form \eqref{1} with $%
B_{1}>0$. Then 
\begin{equation*}
|A_{3}|-|A_{2}|\leq 
\begin{cases}
\frac{{|2B_{1}^{2}-B_{2}|}-3{B_{1}}}{6}, & \text{if }\
|2B_{1}^{2}-B_{2}|\geq 4B_{1}, \\ 
\frac{B_{1}}{6}, & \text{if }\ |2B_{1}^{2}-B_{2}|<4B_{1}.%
\end{cases}%
\end{equation*}%
and 
\begin{equation*}
|A_{3}|-|A_{2}|\geq 
\begin{cases}
\frac{\left\vert 2B_{1}^{2}-B_{2}\right\vert -3B_{1}}{6}, & \text{if }\
B_{1}\geq |4B_{1}^{2}-2B_{2}|, \\[8pt] 
-\frac{B_{1}}{2}\sqrt{\frac{B_{1}}{B_{1}+|2B_{1}^{2}-B_{2}|}}, & \text{if }\
5B_{1}\leq |8B_{1}^{2}-4B_{2}|, \\[8pt] 
-\frac{B_{1}}{6}-\frac{3B_{1}^{2}}{8{|2B_{1}^{2}-B_{2}|}+8{B_{1}}}, & \text{%
otherwise}.%
\end{cases}%
\end{equation*}%
All of these bounds are sharp.
\end{theorem}

\begin{proof}
From \eqref{A2}, we have 
\begin{equation*}
A_{2}=-\frac{B_{1}}{4}p_{1},\quad \quad A_{3}:=\left( \frac{%
2B_{1}^{2}+B_{1}-B_{2}}{24}\right) p_{1}^{2}-\frac{B_{1}}{12}p_{2}.
\end{equation*}%
Now, by using \eqref{a2}, we get 
\begin{equation*}
|A_{3}|-|A_{2}|=|bp_{1}^{2}+cp_{2}|-|ap_{1}|,
\end{equation*}%
where 
\begin{equation*}
a:=\frac{B_{1}}{4},\quad \quad b:=\left( \frac{2B_{1}^{2}+B_{1}-B_{2}}{24}%
\right) ,\quad \text{and}\quad c:=-\frac{B_{1}}{12}.
\end{equation*}%
By utilizing Lemma \ref{lemma1}, we have 
\begin{align*}
|A_{3}|-|A_{2}|& ={|4bt_{1}^{2}+c(2t_{1}^{2}+2(1-{|t_{1}|}^{2})t_{2})|}-2{%
|at_{1}|} \\
& \leq {m|t_{1}|}^{2}+2{c}\left( 1-{|t_{1}|}^{2}\right) |t_{2}|-2{|at_{1}|},
\end{align*}%
with $m:=|4b+2c|$. As we know, $\mathcal{P}$ is rotationally invariant, it
implies, $t_{1}\in \lbrack 0,1]$ and by using the inequality, $|t_{2}|\leq 1$%
, we get 
\begin{equation*}
|A_{3}|-|A_{2}|\leq \left( {m}-2{c}\right) t_{1}^{2}-2{a}t_{1}+2{c}=\Upsilon
(x),
\end{equation*}%
with 
\begin{equation}
\Upsilon (x)=k_{2}x^{2}+k_{1}x+k_{0},  \label{equ12}
\end{equation}%
where 
\begin{equation*}
k_{2}={m}-2{c},\quad \quad k_{1}=-2{a},\quad \quad k_{0}=2{c}.
\end{equation*}%
Now, there are two cases:\newline
$\mathbf{A}_{1}:$ Let $k_{2}\leq 0$. For $k_{1}<0$, so (\ref{equ12})
implies, $\Upsilon ^{\prime }(x)=2k_{2}x+k_{1}<0$ and thus we have that $%
\Upsilon $ is decreasing. Therefore 
\begin{equation*}
\Upsilon (x)\leq \Upsilon (0)=k_{0}=2c=\frac{B_{1}}{6}.
\end{equation*}%
$\mathbf{A}_{2}:$ If $k_{2}>0$, then $\Upsilon $ is a quadratic function
with positive leading coefficient, thus 
\begin{equation*}
\Upsilon (x)\leq \max \{\Upsilon (0),\Upsilon (1)\}.
\end{equation*}%
Here, if $k_{2}+k_{1}<0$, it implies $\Upsilon
(1)=k_{2}+k_{1}+k_{0}<k_{0}=\Upsilon (0)$ and it yields $\Upsilon (x)\leq
\Upsilon (0)=2c$. On the other hand, if $k_{2}+k_{1}\geq 0$, then $\Upsilon
(1)=k_{2}+k_{1}+k_{0}\geq k_{0}=\Upsilon (0)$. Therefore 
\begin{equation*}
\Upsilon (x)\leq \Upsilon (1)=k_{2}+k_{1}+k_{0}=|4b+2c|-2a=\frac{{%
|2B_{1}^{2}-B_{2}|}-3{B_{1}}}{6}.
\end{equation*}%
Thus, from the above discussion, we have 
\begin{equation*}
|A_{3}|-|A_{2}|\leq 
\begin{cases}
\frac{{|2B_{1}^{2}-B_{2}|}-3{B_{1}}}{6}, & \text{if }\
|2B_{1}^{2}-B_{2}|\geq 4B_{1}, \\ 
\frac{B_{1}}{6}, & \text{if }\ |2B_{1}^{2}-B_{2}|<4B_{1}.%
\end{cases}%
\end{equation*}%
The inequality is sharp for the function $f$ defined by \eqref{ma} with $%
p(z)=(1+z)/(1-z)$, when $|2B_{1}^{2}-B_{2}|\geq 4B_{1}$ and for $%
|2B_{1}^{2}-B_{2}|<4B_{1},$ the extremal function $f$ is given by \eqref{ma}
with $p(z)=(1+z^{2})/(1-z^{2})$.\newline
Now, for the lower bound on $|A_{3}|-|A_{2}|$ by using Lemma \ref{lemma1},
we have 
\begin{equation*}
|A_{3}|-|A_{2}|=\left\vert me^{i\theta }t_{1}^{2}+2ce^{i\phi
_{1}}(1-|t_{1}|^{2})t_{2}\right\vert -2|at_{1}|,
\end{equation*}%
with 
\begin{equation*}
m=\left\vert 4b+2c\right\vert ,\quad \quad \theta =arg(4b+2c),\quad \quad
\phi _{1}=arg(2c).
\end{equation*}%
Since $\mathcal{P}$ possesses the property of rotational invariance, thus we
can have, $t_{1}\in \lbrack 0,1]$ and $t_{2}=re^{i\phi _{2}}$ where $r\in
\lbrack 0,1]$. Therefore 
\begin{equation*}
|A_{3}|-|A_{2}|=\left\vert me^{i\theta }t_{1}^{2}+2rce^{i\phi
}(1-t_{1}^{2})\right\vert -2at_{1},
\end{equation*}%
with $\phi =\phi _{1}+\phi _{2}$. As $\left\vert e^{i\phi }\right\vert =1$,
so 
\begin{equation*}
|A_{3}|-|A_{2}|=\left\vert me^{i(\theta -\phi
)}t_{1}^{2}+2rc(1-t_{1}^{2})\right\vert -2at_{1}.
\end{equation*}%
It implies 
\begin{align}
|A_{3}|-|A_{2}|=& \sqrt{m^{2}t_{1}^{4}+4mcrt_{1}^{2}(1-t_{1}^{2})\cos
(\theta -\phi )+4c^{2}r^{2}(1-t_{1}^{2})^{2}}-2at_{1},  \notag \\
\geq & \left\vert mt_{1}^{2}-2rc(1-t_{1}^{2})\right\vert -2at_{1},
\label{equ13}
\end{align}%
since $\cos (\theta -\phi )\geq -1$. Let $b=c=0$. Then 
\begin{equation*}
|A_{3}|-|A_{2}|\geq -2at_{1}\geq -2a=-\frac{B_{1}}{2}.
\end{equation*}%
On the other hand, let $b\neq 0,\ c\neq 0$. Then from \eqref{equ13}, there
are two cases: \newline
\textbf{B}$_{1}$: Let $mt_{1}^{2}-2rc(1-t_{1}^{2})\leq 0$. Then 
\begin{equation*}
t_{1}\leq \sqrt{\frac{2rc}{m+2rc}}:=\zeta _{1}.
\end{equation*}%
By using (\ref{equ13}) and the fact that, $r\in \lbrack 0,1]$, we deduce 
\begin{align*}
|A_{3}|-|A_{2}|& \geq -(m+2rc)t_{1}^{2}-2at_{1}+2rc \\
& \geq -(m+2rc)\zeta _{1}^{2}-2a\zeta _{1}+2rc \\
& =-2a\sqrt{\frac{2rc}{m+2rc}}=-2a\sqrt{\frac{r}{\frac{m}{2c}+r}} \\
& \geq -2a\sqrt{\frac{1}{\frac{m}{2c}+1}} \\
& ={-2a}\sqrt{\frac{2c}{m+2c}} \\
& =-\frac{B_{1}}{2}\sqrt{\frac{B_{1}}{B_{1}+|2B_{1}^{2}-B_{2}|}}.
\end{align*}%
\textbf{B}$_{2}$: Now, let $mt_{1}^{2}-2rc(1-t_{1}^{2})>0$. Then $%
t_{1}>\zeta _{1}$, and we introduce a function $\varphi $ such that 
\begin{equation*}
\varphi (t_{1},r)=(m+2rc)t_{1}^{2}-2at_{1}-2rc.
\end{equation*}%
Now, by using $a^{2}\leq 2c(m+2c)$ and the fact that $r\in \lbrack 0,1]$, we
get 
\begin{equation*}
\frac{\partial \varphi }{\partial t_{1}}=2(m+2r{c})t_{1}-2{a}\geq 2\sqrt{%
2rc(m+2rc)}-2a\geq 0.
\end{equation*}%
It implies, $\varphi $ is an increasing function and by using (\ref{equ13}),
we get 
\begin{equation*}
|A_{3}|-|A_{2}|\geq -2{a}\sqrt{\frac{2r{c}}{m+2r{c}}}\geq -2{a}\sqrt{\frac{2{%
\ c}}{m+2{c}}}=-\frac{B_{1}}{2}\sqrt{\frac{B_{1}}{B_{1}+|2B_{1}^{2}-B_{2}|}},
\end{equation*}%
since $r\in \lbrack 0,1]$. on the other hand, if $a\geq m+2c$, then we get 
\begin{equation*}
\frac{\partial \varphi }{\partial t_{1}}=2(m+2r{c})t_{1}-2{a}\leq 2m+4r{\ c}%
-2{a}\leq 2m+4{c}-2{a}\leq 0.
\end{equation*}%
It implies, $\varphi $ is a decreasing function and from (\ref{equ13}), we
get 
\begin{equation*}
|A_{3}|-|A_{2}|\geq m-2{a}=\frac{\left\vert 2B_{1}^{2}-B_{2}\right\vert
-3B_{1}}{6}.
\end{equation*}%
Lastly, when $a^{2}>2c(m+2c)$ and $a<m+2c$, we get 
\begin{equation*}
\zeta _{1}<\frac{{a}}{m+2{c}}:=\zeta _{2}<1,
\end{equation*}%
and $\frac{\partial \varphi }{\partial r}=-2c(1-t_{1}^{2})\leq 0$ for $r\in
\lbrack 0,1]$ and $t_{1}\in \lbrack \zeta _{1},1]$. Therefore 
\begin{equation*}
\varphi (t_{1},r)\geq \varphi (t_{1},1)=(m+2{c})t_{1}^{2}-2{a}t_{1}-2{c}%
=h(t_{1}).
\end{equation*}%
As $h^{\prime }(t_{1})=0$ implies $t_{1}=\zeta _{2}$. Therefore by using (%
\ref{equ13}), we deduce 
\begin{equation*}
|A_{3}|-|A_{2}|\geq -2{c-\frac{{a}^{2}}{m+2{c}}}=-\frac{B_{1}}{6}-\frac{%
3B_{1}^{2}}{8{|2B_{1}^{2}-B_{2}|}+8{B_{1}}}.
\end{equation*}%
Now, all the above discussion implies 
\begin{equation*}
|A_{3}|-|A_{2}|\geq 
\begin{cases}
\frac{\left\vert 2B_{1}^{2}-B_{2}\right\vert -3B_{1}}{6}, & \text{if }\
B_{1}\geq |4B_{1}^{2}-2B_{2}|, \\[8pt] 
-\frac{B_{1}}{2}\sqrt{\frac{B_{1}}{B_{1}+|2B_{1}^{2}-B_{2}|}}, & \text{if }\
5B_{1}\leq |8B_{1}^{2}-4B_{2}|, \\[8pt] 
-\frac{B_{1}}{6}-\frac{3B_{1}^{2}}{8{|2B_{1}^{2}-B_{2}|}+8{B_{1}}}, & \text{%
otherwise}.%
\end{cases}%
\end{equation*}%
For $B_{1}\geq |4B_{1}^{2}-2B_{2}|$, the equality exits for the function $f$
given by \eqref{ma} with $p(z)=(1+z)/(1-z)$.\newline
Now, Let $5B_{1}\leq |8B_{1}^{2}+4B_{2}|$. Then Lemma \ref{lemma1} implies
that the bound is sharp for the function $f$ given by \eqref{ma}, with $p(z)$
is defined by \eqref{eq001} where 
\begin{equation*}
t_{1}=\sqrt{\frac{{B_{1}}}{{|2B_{1}^{2}-B_{2}|}+{B_{1}}}},\quad \text{and}%
\quad t_{2}=%
\begin{cases}
\frac{2B_{1}^{2}-B_{2}}{{|2B_{1}^{2}-B_{2}|}}, & \text{if }\ 2B_{1}^{2}\neq
B_{2}, \\[1.2em] 
1, & \text{if }\ 2B_{1}^{2}=B_{2}.%
\end{cases}%
\end{equation*}%
Let $2B_{1}^{2}=B_{2}$. Then it is obvious $|A_{3}|-|A_{2}|=-\frac{B_{1}}{2}$%
.\newline
Now, let $2B_{1}^{2}\neq B_{2}$. Then 
\begin{align*}
|A_{3}|-|A_{2}|& =\left\vert bp_{1}^{2}+cp_{{2}}|-|ap_{1}\right\vert \\
& =\frac{1}{6}\left\vert \left( 2B_{1}^{2}-B_{2}\right)
t_{1}^{2}-B_{1}(1-t_{1}^{2})t_{2}\right\vert -\frac{1}{2}\left\vert
B_{1}t_{1}\right\vert
\end{align*}%
and after some calculations, we deduce 
\begin{equation*}
|A_{3}|-|A_{2}|=-\frac{B_{1}}{2}\sqrt{\frac{B_{1}}{B_{1}+|2B_{1}^{2}-B_{2}|}}%
.
\end{equation*}%
At the end, for $5B_{1}>|8B_{1}^{2}-4B_{2}|$ and $B_{1}<|4B_{1}^{2}-2B_{2}|$%
, the extremal function $f$ is given by \eqref{ma}, where $p(z)$ is defined
by \eqref{eq001} with 
\begin{equation*}
t_{1}=\frac{3{B_{1}}}{{|4B_{1}^{2}-2B_{2}|}+2{B_{1}}},\quad \text{and}\quad
t_{2}=\frac{2B_{1}^{2}-B_{2}}{{|2B_{1}^{2}-B_{2}|}}.
\end{equation*}%
Now 
\begin{align*}
36|bp_{1}^{2}+cp_{2}|^{2}& =\left\vert \left( 2B_{1}^{2}-B_{2}\right)
t_{1}^{2}-B_{1}(1-t_{1}^{2})t_{2}\right\vert ^{2} \\
& =\left\vert 2B_{1}^{2}-B_{2}\right\vert ^{2}t_{1}^{4}-2Re\left( B_{1}(%
\overline{2B_{1}^{2}-B_{2}})t_{1}^{2}(1-t_{1}^{2})t_{2}\right)
+B_{1}^{2}(1-t_{1}^{2})^{2} \\
& =\left\vert 2B_{1}^{2}-B_{2}\right\vert ^{2}t_{1}^{4}-2\left\vert
2B_{1}^{3}-B_{1}B_{2}\right\vert
t_{1}^{2}(1-t_{1}^{2})+B_{1}^{2}(1-t_{1}^{2})^{2} \\
& =\left( \left\vert 2B_{1}^{2}-B_{2}\right\vert
t_{1}^{2}-B_{1}(1-t_{1}^{2})\right) ^{2} \\
& =\left( \frac{9B_{1}^{2}}{{|8B_{1}^{2}-4B_{2}|}+4{B_{1}}}-B_{1}\right)
^{2}.
\end{align*}%
As, $5B_{1}>|8B_{1}^{2}-4B_{2}|$, it gives $\frac{9B_{1}^{2}}{{%
|8B_{1}^{2}-4B_{2}|}+4{B_{1}}}-B_{1}>0$, therefore 
\begin{equation*}
|bp_{1}^{2}+cp_{2}|=\frac{3B_{1}^{2}}{8{|2B_{1}^{2}-B_{2}|}+8{B_{1}}}-\frac{%
B_{1}}{6}.
\end{equation*}%
It implies 
\begin{equation*}
|A_{3}|-|A_{2}|=|bp_{1}^{2}+cp_{2}|-|ap_{1}|=-\frac{B_{1}}{6}-\frac{%
3B_{1}^{2}}{8{|2B_{1}^{2}-B_{2}|}+8{B_{1}}}.
\end{equation*}%
It completes the proof.
\end{proof}

\begin{remark}
From Theorem \ref{Th1} and Theorem \ref{Th2}, it is seen that the upper
bounds on $|a_{3}|-|a_{2}|$ and $|A_{3}|-|A_{2}|$ are the same when $%
4B_{1}>\max \left\lbrace |B_{1}^{2}+B_{2}|, \ |2B_{1}^{2}-B_{2}|\right\rbrace $.
\end{remark}

Now we present example of functions in which this property hold.
\begin{example}
Let

$\Phi _{e}(z)=e^{z}=1+z+\frac{1}{2}z^{2}+\cdots ,$

$\Phi _{\sin }(z)=1+\sin (z)=1+z-\frac{1}{6}z^{3}+\cdots ,$

$\Phi _{L}(z)=\sqrt{1+z}=1+\frac{1}{2}z-\frac{1}{8}z^{2}+\cdots ,$

$\Phi _{l}(z)=z+\sqrt{1+z^{2}}=1+z+\frac{1}{2}z^{2}+\cdots ,$

$\Phi _{RL}(z)=\sqrt{2}-(\sqrt{2}-1)\sqrt{\frac{1-z}{1+2(\sqrt{2}-1)z}}=1+%
\frac{5-3\sqrt{2}}{2}z+\frac{71-51\sqrt{2}}{8}z^{2}+\cdots ,$

$\Phi _{C}(z)=1+\frac{4z}{3}+\frac{2z^{2}}{3},$

$\Phi _{\mathcal{BS}}(z)=1+\frac{z}{1-\alpha z^{2}}=1+z+\alpha z^{3}+\cdots $

Then these functions are the subordinating functions for the classes $%
\mathcal{C}_{e},$ $\mathcal{C}_{s},$ $\mathcal{C}_{L},$ $\mathcal{C}_{l},\ 
\mathcal{C}_{RL},\mathcal{C}_{C},$ therefore the upper bounds on $%
|a_{3}|-|a_{2}|$ and $|A_{3}|-|A_{2}|$ are the same for functions in these
classes.
\end{example}

\begin{remark}
	From Theorem \ref{Th1} and Theorem \ref{Th2}, it can also be observed that the lower bounds on $|a_{3}|-|a_{2}|$ and $|A_{3}|-|A_{2}|$ are the same when $ B_{2}=\frac{B_{1}^2}{2}.$
\end{remark}

In the following example we present a few functions in which this property hold.
\begin{example}
Let

$\Phi _{e}(z)=e^{z}=1+z+\frac{1}{2}z^{2}+\cdots ,$

$\Phi _{l}(z)=z+\sqrt{1+z^{2}}=1+z+\frac{1}{2}z^{2}+\cdots.$

Then these functions are the subordinating functions for the classes $%
\mathcal{C}_{e}$ and $\mathcal{C}_{l}$, thus the lower bounds on $%
|a_{3}|-|a_{2}|$ and $|A_{3}|-|A_{2}|$ are same for functions in these
classes.
\end{example}


\begin{thebibliography}{99}
\bibitem{ali} R. Ali, M. Raza, On the difference of coefficients of
non-Bazilevi\v{c} functions, Bull. Iran. Math. Soc. 51(6) (2025) 78.

\bibitem{c2} L. De Branges, A proof of the Bieberbach conjecture, Acta.
Math. 154(1-2) (1985) 137--152.

\bibitem{19} N. E. Cho, B. Kowalczyk, A. Lecko, B. \'{S}miarowska, On the
fourth and fifth coefficients in the Carath\'{e}odory class, Filomat 34(6)
(2020), 2061--2072.


\bibitem{cho} N. E. Cho, Y. J. Sim, D. K. Thomas, On the difference of
coefficients of Bazilevi\v{c} functions, Comput. Methods Funct. Theory 19(4)
(2019) 671--685.

\bibitem{c1} P. L. Duren, Univalent Functions, Grundlehren der
Mathematischen Wissenschaften, Band 259, Springer-Verlag, New York, Berlin,
Heidelberg and Tokyo (1983).


\bibitem{c4} A. Z. Grinspan, Improved bounds for the difference of adjacent
coefficients of univalent functions (Russian), Questions in the mordern
theory of functions (Novosibirsk), Sib. Inst. Mat. 38 (1976) 41--45.

\bibitem{c3} W. K. Hayman, On successive coefficients of univalent
functions, J. London. Math. Soc. 38(1) (1963) 228--243, 1963.

\bibitem{c5} W. K. Hayman, Multivalent Functions, 2nd ed.; Cambridge
University Press: New York, NY, USA, 1994.

\bibitem{lina} L. P. Ilina, The relative growth of nearby coefficients of
schlicht functions. Mat. Zametki 4 (1968), 715--722.



\bibitem{koepf} W. Koepf, On the Fekete--Szego problem for close to convex
functions, Proc. Am. Math. Soc. 101 (1987), 89--95.

\bibitem{leu} Y. Leung, Successive coefficients of starlike functions. Bull.
Lond. Math. Soc. 10 (1978), 193--196.

\bibitem{lib1} R. J. Libera, E. J. Zlotkiewicz, \textit{Early coefficients
of the inverse of a regular convex function}, Proc. Amer. Math. Soc. 85(2)
(1982), 225--230.

\bibitem{11} W. C. Ma, D. Minda, A unified treatment of some special classes
of univalent functions, Proceedings of the Conference on Complex
Analysis(Tianjin, 1992), 157- 169, Conf. Proc. Lecture Notes Anal., I, Int.
Press, Cambridge, MA, 1994.



\bibitem{Ming} L. Ming, T. Sugawa, A note on successive coefficients of
convex functions', Comput. Methods Funct. Theory 17(2) (2017), 179-193.

\bibitem{ob} M. Obradovic, D. K. Thomas, N. Tuneski, On the difference of
coefficients of univalent functions, Filomat 35(11) (2021), 3653-3661.


\bibitem{Raza} U. Raza, M. Raza, On the difference of coefficients of class $%
\mathcal{U}\left( \alpha ,\lambda \right) $, Bull. Malays. Math. Sci. Soc.
48 (2025), 182.


\bibitem{raza3} M. Raza, A. Riaz, D. K. Thomas, Invariance of convex
functions associated with a cardioid domain. Filomat, 39(20) (2025),
6899-6918.


\bibitem{raza4} M. Raza, D. K. Thomas, A. Riaz, Coefficient estimates for
starlike and convex functions related to sigmoid functions, Ukrainian Math.
J. 75 (2023), 782--799


\bibitem{m30} Y. J. Sim, D. K. Thomas, On the difference of inverse
coefficients of univalent functions, Symmetry 12(12) (2020) 2040.

%

\bibitem{m37} D. K. Thomas, N. Tuneski, A. Vasudevarao, Univalent Functions:
A Primer, De Gruyter Stud. Math. 69, Berlin, Boston, 2018.

\bibitem{DK} D. K. Thomas and S. Verma, \textit{Invariance of the
coefficients of strongly convex functions}, Bull. Aust. Math. Soc. 95(3)
(2017), 436--445.
\end{thebibliography}
\end{document}